\documentclass{article}

\usepackage{amsmath,amssymb,amsthm}
\usepackage[normalem]{ulem}

\newcommand{\mathset}[1]{\left\{#1\right\}}

\newtheorem{theorem}{Theorem}
\newtheorem{lem}{Lemma}

\DeclareMathOperator\closure{cl}

\begin{document}
\title{
 Dimension of Alexandrov Topologies
}
\author{Patrick Erik Bradley and Norbert Paul}

\date{\today}
\maketitle

\begin{abstract}
We prove that the Krull dimension of an  Alexandrov space of finite height
can be characterised 
with the specialisation preorder of its associated $T_0$-space.
\end{abstract}

A topological space is called \emph{Alexandrov}, if the intersection of an arbitrary family of closed subsets is closed.
It is well known that an Alexandrov topology on a set $X$ is determined by
the \emph{specialisation preorder} $\le$ given as:
$$
x\le y:\Leftrightarrow x\in\closure\mathset{y},
$$
and it is well known that the specialisation 
preorder is a partial order if and only if $X$ is a $T_0$-space.

A  topological space $X$ is \emph{irreducible}, if
it is not the union of two proper, non-empty, closed subsets $B,C$:
$$
X=B\cup C\Rightarrow X=B  \;\textrm{or}\; X=C.
$$

For a topological space $X$, there is the \emph{Krull dimension}, defined as 
the supremum of all lengths $n$ of chains
$$
X_0\subset X_1\subset \dots\subset X_n
$$
with proper inclusions of non-empty closed irreducible subsets.
The Krull dimension of the empty space $\emptyset$ is defined as $-1$.
A space is called  \emph{finite dimensional},
if  its Krull dimension is 
a finite number.

\bigskip
For an Alexandrov space $X$, there is the \emph{height}, defined as the supremum
of all lengths $n$ of chains
$$
\closure\mathset{x_0}\subset\dots\subset\closure\mathset{x_n}
$$
with proper inclusions. 
In  case $X$ is a $T_0$-Alexandrov space, then the height coincides with
the supremum of all lengths of  chains
$$
x_0\le \dots\le x_n
$$
where $x_i\neq x_{i+1}$ for $i=0,\dots,n-1$.
An Alexandrov space is said to be \emph{of finite height}, 
if its height is a finite number.

\bigskip
The height and the Krull dimension of an Alexandrov space are related, because
of the following Lemma which holds true in any topological space:

\begin{lem}\label{closureofptisirred}
Let $X$ be a topological space. Then $\closure\mathset{x}$ is irreducible for any
$x\in X$.
\end{lem}

\begin{proof}
Notice that $\closure\mathset{x}$ 
is the smallest closed subset of $X$ containing $x$. Hence, if 
$$
\closure\mathset{x}=B\cup C
$$
is the union of closed subsets $B,C$ of $\closure\mathset{x}$ 
in the subspace topology
with 
$B\not\subseteq C$ and $C\not\subseteq B$, then first of all, 
$B$ and $C$ are also closed in $X$, 
because $\closure\mathset{x}$ is closed in $X$, and closedness is transitive:
$F$ closed subset of $G$, and $G$ closed subset of $H$ implies that 
$F$ is a closed subset of $H$. 
Consequently, $x$ cannot lie in both $B$ and $C$, because $\closure\mathset{x}$
is the smallest closed subset of $X$ containing $a$.
Hence, we may assume that $x\in B$ and $x\notin C$.
But then $x$ 
is contained in the proper subset $B$ of $\closure\mathset{x}$ 
which is closed in $X$.
This cannot be, because $\closure\mathset{x}$ is the smallest closed
subset of $X$ containing $a$. 
This proves that $\closure\mathset{x}$ is irreducible.
\end{proof}

If the height of an Alexandrov space is finite, then there is a simple
characterisation of irreducible subsets.

\begin{lem}\label{irred=clpt}
A closed subset $A$ of an  Alexandrov space
$X$ of finite height is irreducible if and only if 
$$
A=\closure\mathset{a},
$$
i.e.\ $A$ is the closure of a point $a\in A$.
\end{lem}

\begin{proof}
In Lemma \ref{closureofptisirred}, it was shown that $\closure\mathset{a}$
is irreducible.

Let the closed set $A$ be irreducible.
If $A\neq\closure\mathset{x}$ for any $x\in A$, then 
$$
A=\closure\left(A\setminus\closure\mathset{x}\right)\cup\closure\mathset{x}
$$
is the union of two proper non-empty closed subsets. Hence, for all $x\in A$:
$$
A=\closure\left(A\setminus\closure\mathset{x}\right),
$$
because $A\neq\closure\mathset{x}$ and $A$ is irreducible.
Since $x\in A$, it follows that the open set $U_x:=\mathset{u\in X\mid x\le u}$
has a non-empty intersection with $A\setminus\closure\mathset{x}$. In other words,
there exists  $y\in A\setminus\closure\mathset{x}$ such that $y\ge x$.
Now $A\setminus\closure\mathset{x}$  contains  
a maximal element
$a$ such that $a\ge x$. It is maximal in the sense that
for all  $b\in A\setminus\closure\mathset{x}$
with $b\ge a$ it holds true that
 $\closure\mathset{b}=\closure\mathset{a}$.
Otherwise there would be an infinite
 ascending chain
$$
\closure\mathset{a}\subset\closure\mathset{b}\subset\closure\mathset{c}\subset\dots
$$
with strict inclusions, where 
$a,b,c,\ldots\in A\setminus\closure\mathset{x}$.
This contradicts the finiteness of the height of $X$.
So, from 
$$
A=\closure(A\setminus\closure\mathset{a}),
$$
it follows that there is 
$b\in A\setminus\closure\mathset{a}\subseteq A\setminus\closure\mathset{x}$
such that $b\ge a$. Hence, by the maximality property of $a$, 
it follows that $\closure\mathset{a}=\closure\mathset{b}$ which cannot be,
as otherwise $b\in A\setminus\closure\mathset{b}$.
\end{proof}

Let $X$ be a topological space which is
 Alexandrov with specialisation preorder $\le$. 
There is a natural equivalence relation $\sim$ on $X$:
$$
x\sim y:\Leftrightarrow x\le y\;\textrm{and}\; y\le x,
$$
and let $X_0:=X/\sim$ be the Kolmogorov quotient.
It is a $T_0$-space,  its induced
specialisation preorder
is a partial order. 
There is a canonical map $\pi\colon X\to X_0$ which takes 
each $x\in X$ to its equivalence class.
It has the important property:
\begin{equation} \label{importantproperty}
x\le y\Leftrightarrow \pi(x)\le\pi(y)
\end{equation}

\begin{lem}\label{pitakescl2cl}
The map $\pi\colon X\to X_0$ induces a bijection between  sets
\begin{equation}\label{thebijection}
\mathset{\closure\mathset{x}\mid x\in X}
\cong
\mathset{\closure\mathset{x_0}\mid x_0\in X_0}
\end{equation}
\end{lem}

\begin{proof}
First, observe that $\pi(\closure\mathset{x})=\closure\mathset{\pi(x)}$. 
Namely, $y\in\closure\mathset{x}$ implies $\pi(y)\in\closure\mathset{\pi(x)}$
by continuity, and
$\pi(y)\in\closure\mathset{\pi(x)}$ implies $\pi(y)\le\pi(x)$,
from which it follows by (\ref{importantproperty}) 
that $y\le x$, i.e.\ $x\in\closure\mathset{y}$.

Secondly, observe that $\pi^{-1}(\closure\mathset{\pi(x)})=\closure\mathset{x}$.
Here, the inclusion $\supseteq$ is clear, because the left hand side is closed
and contains $x$. For the other inclusion $\subseteq$: 
$y\in\pi^{-1}(\closure\mathset{\pi(x)})$ implies $\pi(y)\le \pi(x)$,
hence $y\le x$ by (\ref{importantproperty}), i.e.\ $y\in\closure\mathset{x}$.

By those two observations, we have established the bijection
(\ref{thebijection}).
\end{proof}

\begin{theorem}
An Alexandrov space $X$ is finite dimensional if and only if it is
of finite height. In this case, the Krull dimension and the height of $X$
coincide.
Furthermore, this number equals the height and the Krull dimension
 of the Kolmogorov
quotient of $X$.
\end{theorem}

\begin{proof}
Clearly, if $X$ is of finite dimension, then $X$ is of finite height.
The converse assertion is an immediate consequence of Lemma \ref{irred=clpt},
from which it also follows that the Krull dimension and height coincide, if
they are finite.
The last assertion follows from the  
fact that the height of $X$ equals
the height of $X_0$, and that the Krull dimension of $X$ equals the Krull 
dimension of $X_0$. This latter statement follows from 
Lemma \ref{pitakescl2cl}.
\end{proof}

\end{document}